\documentclass[11pt]{amsart}
\usepackage{amsmath}
\usepackage{amscd}
\usepackage{amsthm}
\usepackage{verbatim}
\usepackage{fullpage}
\usepackage{amssymb}
\usepackage{amsfonts}
\usepackage{mathrsfs}
\usepackage{stmaryrd}
\usepackage[usenames]{color}
\usepackage{graphicx}
\usepackage{hyperref}
\usepackage{mathscinet}
\usepackage{palatino}
\usepackage{bbm}
\usepackage{latexsym}
\usepackage{exscale}
\usepackage[normalem]{ulem}
\usepackage{quiver}
\usepackage{enumitem}

\usepackage[backend=biber,style=alphabetic,doi=false,isbn=false,url=false,minnames=6,maxnames=6,maxalphanames=6]{biblatex}
\renewbibmacro{in:}{%
  \ifentrytype{article}{}{\printtext{\bibstring{in}\intitlepunct}}}
\renewbibmacro*{volume+number+eid}{%
  \printtext{vol.}
  \printfield{volume}%
  \setunit*{\addnbspace}
  \printfield{number}%
  \setunit{\addcomma\space}%
  \printfield{eid}}
\DeclareFieldFormat[article]{number}{\mkbibparens{#1}}

\newcommand{\doi}[1]{\href{https://dx.doi.org/#1}{{\small DOI:#1}}}

\newcommand{\googlebooks}[1]{(preview at \href{https://books.google.com/books?id=#1}{google books})}

\newcommand{\numdam}[1]{}

\theoremstyle{plain}

\newtheorem{proposition}{Proposition}
\newtheorem{prop}[proposition]{Proposition}

\newtheorem{lemma}[proposition]{Lemma}
\newtheorem{cor}[proposition]{Corollary}

\numberwithin{equation}{section}
\newtheorem{example}[proposition]{Example}

\newtheorem{remark}[proposition]{Remark}

\newtheorem{defi}[proposition]{Definition}    
    
\newtheorem*{thmnono}{Theorem}

\usepackage[all]{xy}
\usepackage{graphicx}
\usepackage{tikz}
\usepackage{tikz-cd}
\usetikzlibrary{calc}
\usetikzlibrary{decorations.markings}
\usetikzlibrary{decorations.pathreplacing}
\usetikzlibrary{arrows,shapes,positioning}
\usetikzlibrary{tqft}
\tikzstyle directed=[postaction={decorate,decoration={markings,
    mark=at position #1 with {\arrow{>}}}}]
\tikzstyle rdirected=[postaction={decorate,decoration={markings,
    mark=at position #1 with {\arrow{<}}}}]
\tikzset{anchorbase/.style={baseline={([yshift=-0.5ex]current bounding box.center)}},
  tinynodes/.style={font=\tiny,text height=0.75ex,text depth=0.15ex},
  smallnodes/.style={font=\scriptsize,text height=0.75ex,text depth=0.15ex},
  >={Latex[length=1mm, width=1.5mm]}
}

\tikzset{
    partial ellipse/.style args={#1:#2:#3}{
        insert path={+ (#1:#3) arc (#1:#2:#3)}
    }
}

\def\Z{\mathbb{Z}}

\def\C{\mathbb{C}} 
 
\def\N{\mathbb{N}}

\DeclareMathOperator{\id}{id}

\definecolor{kwcolor}{rgb}{.05, .5, .3}
\definecolor{pwcolor}{rgb}{.15, .5, .5}
\definecolor{dred}{rgb}{.7, 0, 0}

\newcommand\figureXv{
	\begin{tikzpicture}
		\coordinate (O) at (0,0.1);
		\coordinate (A) at (-1/9,1/2);
		\coordinate (B) at (0,0.4);
		\coordinate (C) at (-1/9,0);
		
		\draw[style=thick] (O)--(A);
		\draw[style=thick,color=black] (B) to  (C);
	\end{tikzpicture}
}

 \title{A note on TQFTs for orientable 2-dimensional cobordisms} 
\author{Leon J. Goertz}
\address{Fachbereich Mathematik, Universit\"at Hamburg, 
Bundesstra{\ss}e 55, 
20146 Hamburg, Germany
}
\email{leon.goertz@uni-hamburg.de}

\author{Paul Wedrich}
\address{Fachbereich Mathematik, Universit\"at Hamburg, 
Bundesstra{\ss}e 55, 
20146 Hamburg, Germany
\href{https://paul.wedrich.at}{paul.wedrich.at}}
\email{paul.wedrich@uni-hamburg.de}

\begin{document}
\begin{abstract} Topological quantum field theories (TQFTs) are symmetric monoidal functors out of cobordism categories. In dimension two, oriented TQFTs are famously classified by commutative Frobenius algebras. In the unoriented setting, the classification requires additional data: an involution and a value assigned to the Möbius strip. In this work, we describe an intermediate framework that classifies 2-dimensional TQFTs for orientable cobordisms, in an appropriate sense. Our motivation arises from skein-theoretic models of surfaces embedded in 3-manifolds and Khovanov homology, where surfaces are often treated as unoriented, even though the associated 2-dimensional TQFTs themselves need not be fully unoriented.
\end{abstract}
\maketitle

\newcommand{\V}{\mathcal{V}}
\newcommand{\Cob}{\mathrm{Cob}_2}
\newcommand{\UCob}{\mathrm{UCob}_2}
\newcommand{\OblCob}{\mathrm{OCob}_2}
\newcommand{\OblCobq}{\mathrm{OCob}^{\mathrm{quot}}_2}

\section{Results}
We let $\Cob$ denote the symmetric monoidal category of oriented 2-dimensional cobordisms between oriented closed 1-manifolds and $\UCob$ the symmetric monoidal category of (unoriented) 2-dimensional cobordisms between (unoriented) closed 1-manifolds. 

\begin{defi} 
\label{def:OblCob}
The category $\OblCob$ is defined as the \emph{non-full} subcategory of $\UCob$, containing all objects and only those morphisms that are composites of \emph{orientable} cobordisms.
\end{defi}

\begin{example}
    \label{exa:Klein}
    The Klein bottle, although unorientable, is a morphism of $\OblCob$, since it can be obtained as a composite of orientable cobordisms. In particular, the set of orientable cobordisms is not closed under composition and hence does not form a subcategory of $\UCob$. Its closure under composition is $\OblCob$.
\end{example}

Our goal is to give a presentation of $\OblCob$ that places it as an intermediate object between $\Cob$ and $\UCob$ and fully classifies symmetric monoidal functors out of it, i.e. \emph{orientable TQFTs}.

\begin{thmnono} As symmetric monoidal categories, the following cobordism categories are freely generated by:
\begin{itemize}\setlength{\itemindent}{-1em}
    \item $\Cob$: a commutative Frobenius algebra \cite{MR1414088,MR2037238}.
    \item $\OblCob$: a commutative Frobenius algebra with involution; see Proposition~\ref{prop:OblCob}.
    \item $\UCob$: a commutative Frobenius algebra with involution and compatible Möbius morphism \cite{MR2253441}.
\end{itemize}
\end{thmnono}
We will recall these characterizations in the oriented and unoriented cases in Propositions~\ref{prop:Cob} and \ref{prop:UCob}, respectively, and then use them to deduce the additional intermediate case. Let $\V$ be a symmetric monoidal category serving as target for our TQFTs.

\begin{cor} Evaluation at the circle $S^1$ constitutes an equivalence of symmetric monoidal categories
\[
\mathrm{Fun}^{\otimes}(\OblCob, \V) \xrightarrow{\cong} \mathrm{InvFrobAlg}(\V)
\]
from the symmetric monoidal functors $Z\colon \OblCob\to \V$ to the category $\mathrm{InvFrobAlg}(\V)$, whose objects are the involutive Frobenius algebra objects in $\V$, i.e. Frobenius algebra objects $A$ in $\V$ equipped with a Frobenius algebra morphism $\phi\colon A \to A$ such that $\phi^2=\id_A$, and whose morphisms are Frobenius algebra morphisms that intertwine the involutions.
\end{cor}

To relate TQFTs out of $\Cob$, $\OblCob$, and $\UCob$, we consider the following diagram.

\begin{equation}
    \label{eq:diag}
\begin{tikzcd}
	{\Cob} \\
	& {\OblCob} && {\mathcal{V}} \\
	{\UCob}
	\arrow["F_1"{description},hook, shift left=0,from=1-1, to=2-2]
	\arrow[from=1-1, to=2-4]
    \arrow[from=2-2, to=2-4]
	\arrow["F"{description},from=1-1, to=3-1]
	\arrow[curve={height=-8pt},"G_1"{description},shift left=0, dashed, from=2-2, to=1-1]
	\arrow["F_2"{description},hook, from=2-2, to=3-1]
	\arrow[from=3-1, to=2-4]
\end{tikzcd}
\end{equation}

Here, the left vertical arrow $F$ is the forgetful functor which forgets the orientations of 1-manifolds and 2-dimensional cobordisms. It factors through $\OblCob$, by construction. The arrows to $\V$ indicate TQFTs of the appropriate types. The functor $G_1$ will be constructed in Proposition~\ref{prop:G} and shows that every oriented 2-dimensional TQFT can be promoted to one defined on $\OblCob$ by choosing the identity as Frobenius algebra involution. 

\begin{remark}
    The functors $F, F_1, F_2$ are essentially surjective and none of them is full, although $F_2$ is faithful by definition. That $F_2$ is not full can be seen in Corollary~\ref{cor:crosscap}. Non-fullness of $F_1$ was already observed in Example~\ref{exa:Klein} and faithfulness follows from Proposition~\ref{prop:G}.
\end{remark}

The results of this note are almost certainly known to experts. Our motivation was to justify the use of commutative Frobenius algebras in the construction of skein theories for 3-manifolds based on \emph{unoriented} embedded surfaces labeled by elements of commutative Frobenius algebras---the Asaeda--Frohman--Kaiser TQFTs \cite{AsaedaFrohman, KaiserFrobAlg}---and in Bar-Natan's description \cite{MR2174270} of Khovanov homology \cite{MR1740682}. The Frobenius algebra $H^*(\C P^1)$, which is ubiquitous in these constructions, does \emph{not} provide an unoriented TQFT defined on $\UCob$ \cite{czenky2025extendedfrobeniusstructures}, but one on $\OblCob$, which is sufficient. We refer to \cite{leonsurfaceskein} for a more comprehensive survey of related literature.

\subsection*{Acknowledgements and Funding}
We thank Agustina Czenky and David Reutter for useful discussions. The authors acknowledge support from the Deutsche Forschungsgemeinschaft (DFG, German Research Foundation) under Germany's Excellence Strategy - EXC 2121 ``Quantum Universe'' - 390833306 and the Collaborative Research Center - SFB 1624 ``Higher structures, moduli spaces and integrability'' - 506632645.

\section{Proofs}

The cobordism categories $\Cob$, $\UCob$, $\OblCob$ are symmetric monoidal with respect to disjoint union and as such have objects generated by the standard circle $S^1$ (say, the unit circle in the complex plane). For $\Cob$ this is considered as equipped with the positive orientation. The standard circle equipped with the negative orientation, denoted $\overline{S^1}$, is orientation-preservingly diffeomorphic to $S^1$, e.g. by the restriction of complex conjugation. We thus may study $\Cob$, $\UCob$, $\OblCob$ as PROPs, i.e. product and permutation categories \cite{MR170925}, with objects (up to isomorphism) indexed by the natural numbers $n\in \N_0$, on which the monoidal structure acts by addition.
\smallskip

In the following we introduce notation and graphical conventions for some standard cobordisms. We read graphics from right to left, compatible with the usual convention for composing functions\footnote{The following graphics are imported from \cite{czenky2024unoriented2dimensionaltqftscategory}, but with left-right mirrored.}:
\begin{align}\label{unoriented generators}
		\begin{tikzpicture}[tqft/cobordism/.style={draw,thick}, anchorbase,
			tqft/view from=outgoing, tqft/boundary separation=30pt,
			tqft/cobordism height=40pt, tqft/circle x radius=8pt,
			tqft/circle y radius=4.5pt, tqft/every boundary component/.style={rotate=90}
			]
			\pic[tqft/cylinder,rotate=90,name=a,anchor={(1,0)}, every incoming
			boundary component/.style={draw,dotted,thick},every outgoing
			boundary component/.style={draw,thick}
			];
			\node at ([yshift=-18pt,xshift=-20pt]a-outgoing boundary 1){\small{Id}};
            \end{tikzpicture}
            \;,\; 
            \begin{tikzpicture}[tqft/cobordism/.style={draw,thick}, anchorbase,
			tqft/view from=outgoing, tqft/boundary separation=30pt,
			tqft/cobordism height=40pt, tqft/circle x radius=8pt,
			tqft/circle y radius=4.5pt, tqft/every boundary component/.style={rotate=90}
			]
			\pic[tqft/pair of pants,
			rotate=90,name=b,at=(a-outgoing boundary), every incoming
			boundary component/.style={draw,dotted,thick},every outgoing
			boundary component/.style={draw,thick}];
			\node at ([yshift=-18pt,xshift=-20pt]b-outgoing boundary 1){\small{$m$}};
            \end{tikzpicture}
             \;,\; 
            \begin{tikzpicture}[tqft/cobordism/.style={draw,thick}, anchorbase,
			tqft/view from=outgoing, tqft/boundary separation=30pt,
			tqft/cobordism height=40pt, tqft/circle x radius=8pt,
			tqft/circle y radius=4.5pt, tqft/every boundary component/.style={rotate=90}
			]
			\pic[tqft/cup, 
			rotate=90,name=d,anchor={(1,-0.5)},every incoming
			boundary component/.style={draw,dotted,thick},every outgoing
			boundary component/.style={draw,thick}];
			\node at ([yshift=-18pt,xshift=2pt]d-incoming boundary 1){\small{$u$}};
            \end{tikzpicture}
            \;,\;  
            \begin{tikzpicture}[tqft/cobordism/.style={draw,thick}, anchorbase,
			tqft/view from=outgoing, tqft/boundary separation=30pt,
			tqft/cobordism height=40pt, tqft/circle x radius=8pt,
			tqft/circle y radius=4.5pt, tqft/every boundary component/.style={rotate=90}
			]
			\pic[tqft/reverse pair of pants, name=c,rotate=90,at=(b-outgoing boundary),every incoming
			boundary component/.style={draw,thick,dotted},every outgoing
			boundary component/.style={draw,thick}];
			\node at ([yshift=-18pt,xshift=20pt]c-incoming boundary 1){\small{$\Delta$}};
            \end{tikzpicture}
             \;,\;  
            \begin{tikzpicture}[tqft/cobordism/.style={draw,thick}, anchorbase,
			tqft/view from=outgoing, tqft/boundary separation=30pt,
			tqft/cobordism height=40pt, tqft/circle x radius=8pt,
			tqft/circle y radius=4.5pt, tqft/every boundary component/.style={rotate=90}
			]
			\pic[tqft/cap, 
			rotate=90,name=e,anchor={(1,-4.3)}, every incoming
			boundary component/.style={draw,dotted,thick},every outgoing
			boundary component/.style={draw,thick}];
			\node at ([yshift=-18pt,xshift=-2pt]e-outgoing boundary 1){\small{$\varepsilon$}};
            \end{tikzpicture}
             \;,\; 
            \begin{tikzpicture}[tqft/cobordism/.style={draw,thick}, anchorbase,
			tqft/view from=outgoing, tqft/boundary separation=30pt,
			tqft/cobordism height=40pt, tqft/circle x radius=8pt,
			tqft/circle y radius=4.5pt, tqft/every boundary component/.style={rotate=90}
			]
			\pic[tqft/cylinder to next,rotate=90,name=k, anchor=incoming boundary 1,boundary separation=60pt, anchor={(1.07,-5.8)},every incoming
			boundary component/.style={draw,dotted,thick},every outgoing
			boundary component/.style={draw,thick}];
			\pic[tqft/cylinder to prior,rotate=90, anchor={(0.6,-5.8)},boundary separation=60pt,every incoming
			boundary component/.style={draw,dotted,thick},every outgoing
			boundary component/.style={draw,thick}];
			\node at ([yshift=-50pt,xshift=-18pt]k-outgoing boundary 1){\small{$\tau$}};
            \end{tikzpicture}
             \;,\;  
            \begin{tikzpicture}[tqft/cobordism/.style={draw,thick}, anchorbase,
			tqft/view from=outgoing, tqft/boundary separation=30pt,
			tqft/cobordism height=40pt, tqft/circle x radius=8pt,
			tqft/circle y radius=4.5pt, tqft/every boundary component/.style={rotate=90}
			]
            \pic[tqft/cylinder,rotate=90,name=k,anchor={(1,-7.3)},every incoming
			boundary component/.style={draw,dotted,thick},every outgoing
			boundary component/.style={draw,thick}];
			\node at ([yshift=-18pt,xshift=-20pt]k-outgoing boundary 1){\small{$\phi$}};
			\node at ([xshift=20pt]k-incoming boundary 1){$\leftrightarrow$};
            \end{tikzpicture}
             \;,\;  
            \begin{tikzpicture}[tqft/cobordism/.style={draw,thick}, anchorbase,
			tqft/view from=outgoing, tqft/boundary separation=30pt,
			tqft/cobordism height=40pt, tqft/circle x radius=8pt,
			tqft/circle y radius=4.5pt, tqft/every boundary component/.style={rotate=90}
			]
			\pic[tqft/cap, 
			rotate=270,name=d,at=(a-outgoing boundary),every incoming
			boundary component/.style={draw,dotted,thick},every outgoing
			boundary component/.style={draw,thick}];
			\node at ([xshift=7pt]d-outgoing boundary 1) {$\figureXv$};
			\node at ([xshift=10pt]d-outgoing boundary 1) {.};
			\node at ([yshift=-18pt,xshift=2pt]d-outgoing boundary 1){\small{$\theta$}};
		\end{tikzpicture}
	\end{align}
    The first five are well-known. The transposition $\tau$ is useful when describing cobordism categories as freely generated as monoidal categories and for expressing the (co)commutativity of the (co)multiplication, see below. When generating as symmetric monoidal category, $\tau$ does not need to be named as generator. The cobordism $\phi$ is a cylinder over the circle $S^1$, but with orientation-reversing source- and target identifications. Note that it is an orientable cobordism. The symbol $\theta$ refers to the Möbius band (or punctured $\mathbb{R}P^2$), considered as an unorientable cobordism from the empty 1-manifold to the circle $S^1$.
    \smallskip

We recall the presentation of $\Cob$ underlying the famous classification of oriented 2-dimensional TQFTs \cite{MR1414088,MR2037238}.
\begin{prop}
\label{prop:Cob} The oriented cobordism category $\Cob$ is generated as a PROP by the morphisms $(m,u,\Delta,\varepsilon)$ subject to the relations:
\begin{itemize}
\item The morphisms $(m,u)$ define a unital, associative, and commutative algebra:
\begin{equation}
\label{eq:1}
m\circ(u\otimes \id)=\id=m\circ(\id\otimes u)\;,\quad
m\circ(m\otimes \id)=m\circ(\id\otimes m)\;,\quad
m\circ \tau=m
\end{equation}
\item The morphisms $(\Delta,\varepsilon)$ define a counital, coassociative, and cocommutative coalgebra:
\begin{equation}
\label{eq:2}
(\varepsilon\otimes \id)\circ \Delta=\id=(\id\otimes \varepsilon)\circ\Delta\;,\quad
(\Delta\otimes \id)\circ \Delta=(\id\otimes \Delta)\circ\Delta\;,\quad
\tau\circ \Delta=\Delta
\end{equation}
\item The Frobenius relations hold:
\begin{equation}
\label{eq:3}
(\id \otimes m) \circ (\Delta\otimes \id) = \Delta \circ m = (m\otimes \id)\circ (\id \otimes \Delta)
\end{equation}
In other words, $\Cob$ is the PROP freely generated by a commutative Frobenius algebra object.
\end{itemize}
\end{prop}

Illustrations of these and following relations can be found in \cite{MR2253441,czenky2024unoriented2dimensionaltqftscategory}.

\begin{prop}
\label{prop:UCob}
    The unoriented cobordism category $\UCob$ is generated as a PROP by the morphisms $(m,u,\Delta,\varepsilon,\phi, \theta)$ subject to the relations \eqref{eq:1}, \eqref{eq:2}, \eqref{eq:3} and additionally:
    \begin{itemize}
        \item The morphism $\phi$ is a Frobenius algebra involution:
        \begin{equation}
        \label{eq:4}
   \phi^2=\id \;,\quad
    m\circ(\phi\circ\phi)=\phi\circ m\;,\quad
   \phi\circ u = u\;,\quad
    \Delta \circ\phi = (\phi\otimes\phi)\circ \Delta\;,\quad
    \varepsilon\circ\phi=\varepsilon
    \end{equation}
    \item The two extended Frobenius relations hold:
    \begin{align}
        \label{eq:5}
        m\circ (\theta\otimes \id) &= \phi \circ (m\circ (\theta\otimes \id))\\
        \label{eq:6}
        m\circ (\theta\otimes \theta) &=  m\circ (\phi\otimes \id) \circ \Delta \circ u
    \end{align}
    \end{itemize}
\end{prop}
\begin{proof}
    As explained in \cite[Section 3.1]{czenky2024unoriented2dimensionaltqftscategory}, this follows from \cite[Section 2.2]{MR2253441}.
\end{proof}
Related results in other contexts appear in \cite{MR1424641, MR2305607, MR2026879}. 

\begin{example}
    The expression $m\circ (\phi\otimes \id) \circ \Delta \circ u$ on the right-hand side of \eqref{eq:6} presents a punctured Klein bottle as composite of orientable cobordisms, as stated in Example~\ref{exa:Klein}.
\end{example}

\begin{lemma}\label{lem:charobl} Let $\Sigma$ be a 2-dimensional cobordism in $\UCob$. The following conditions on $\Sigma$ are equivalent:
\begin{enumerate}[label = (\arabic*)]
    \item $\Sigma$ is contained in the subcategory $\OblCob$.
    \item For every connected component $C\in \pi_0(\Sigma)$, we have $\chi(C)-|\pi_0(\partial C)| \in 2\Z$.
    \item In the presentation of Proposition~\ref{prop:UCob}, every connected component $C\in \pi_0(\Sigma)$ contains an even number of crosscaps.
    \item In the presentation of Proposition~\ref{prop:UCob}, the cobordism $\Sigma$ can be written as a composite without crosscaps.
\end{enumerate}
\end{lemma}
\begin{proof}
    Note that the statistic $X(-):=\chi(-)-|\pi_0(\partial -)|$ is additive under disjoint union and evaluates to an even number for every orientable cobordism. In particular, it does so for every connected component of an orientable cobordism. The parity is, furthermore, additive under composition, so that every morphism of $\OblCob$ satisfies (2). In Proposition~\ref{prop:UCob} the crosscap is the only generator with odd statistic $X$ and so the parity of $X$ equals the parity of the number of crosscaps. We conclude that (2) implies (3). If a connected component contains a non-zero even number of crosscaps, these can be rewritten in pairs as punctured Klein bottles by relation \eqref{eq:6}, so that (3) implies (4). Since all generators except the crosscap are orientable, (4) implies (1). 
\end{proof}

\begin{cor}\label{cor:crosscap}
The crosscap is not in $\OblCob$ and hence $F_2\colon \OblCob \to \UCob$ is not full.
\end{cor}
\begin{proof}
    The parity of $X$ is odd for the crosscap, so it is not in $\OblCob$ by Lemma~\ref{lem:charobl}.
\end{proof}

We can now provide a presentation for $\OblCob$. Here we write $K:=m\circ (\phi\otimes \id) \circ \Delta \circ u$ for the punctured Klein bottle.
\begin{prop}
\label{prop:OblCob}
The orientable cobordism category $\OblCob$ is generated as a PROP by the morphisms $(m,u,\Delta,\varepsilon,\phi)$ subject to only\footnote{In particular, no analogue of relation \eqref{eq:5} is required.} the relations \eqref{eq:1}, \eqref{eq:2}, \eqref{eq:3}, \eqref{eq:4}.
\end{prop}
\begin{proof}
   The morphisms $(m,u,\Delta,\varepsilon,\phi)$ generate $\OblCob$ by Lemma~\ref{lem:charobl}.(4). Next, we recall that every pair of crosscaps on a connected component can be re-expressed by \eqref{eq:6}, which now reads $m\circ (\theta\otimes \theta)=K$ as punctured Klein bottle. The only other relation in $\UCob$ concerning the crosscap is \eqref{eq:5}, which, informally stated, says that an orientation reversal near a boundary component can be absorbed in the presence of a crosscap on the same component. However, according to Lemma~\ref{lem:charobl}.(3), any morphism in $\OblCob$ where such relation could be applied has at least two crosscaps on the same component, i.e. punctured Klein bottle. The remnant of \eqref{eq:5} in $\OblCob$ is, thus, the relation:
    \begin{equation}
        \label{eq:7}
        m\circ (K\otimes \id) = \phi \circ (m\circ (K\otimes \id))
   \end{equation}
    To prove the proposition, we show that \eqref{eq:7} is already a consequence of \eqref{eq:1}--\eqref{eq:4}. We start by re-expressing the twice-punctured Klein bottle:
    \begin{align*}
    m\circ (K\otimes \id)
    &\overset{def}{=} m\circ ((m\circ (\phi\otimes \id) \circ \Delta \circ u))\otimes \id)\\
    &=
    m \circ (m\otimes \id) \circ (\phi \otimes \id^2) \circ (\Delta \otimes \id) \circ (u\otimes \id)\\
    &\overset{a}{=}
    m \circ (\id\otimes m) \circ (\phi \otimes \id^2) \circ (\Delta \otimes \id) \circ (u\otimes \id)\\
    &\overset{\otimes}{=}
    m \circ (\phi \otimes \id) \circ (\id\otimes m) \circ (\Delta \otimes \id) \circ (u\otimes \id)\\
    &\overset{F}{=}
    m \circ (\phi \otimes \id) \circ \Delta \circ m \circ (u\otimes \id)\\
    &\overset{u}{=}
    m \circ (\phi \otimes \id) \circ \Delta 
    \end{align*}
    Post-composing the expression $m \circ (\phi \otimes \id) \circ \Delta $ instead of  $m\circ (K\otimes \id)$ with $\phi$ now yields:
    \begin{align*}
    \phi \circ (m \circ (\phi \otimes \id) \circ \Delta )
    &\overset{\phi}{=} m \circ (\phi^2 \otimes \phi) \circ \Delta
    \\
      &\overset{\phi}{=}m \circ (\id \otimes \phi) \circ \Delta \\
     &\overset{\tau}{=}
     m \circ \tau\circ (\phi \otimes \id)\circ \tau \circ \Delta \\
      &\overset{m,\Delta}{=}
     m \circ (\phi \otimes \id)\circ \Delta \qedhere
    \end{align*}
\end{proof}

\begin{prop}\label{prop:G}
    A symmetric monoidal functor $G_1\colon \OblCob \to \Cob$ is uniquely and well-defined by the assignment
    \[S^1\mapsto S^1\;,\quad (m,u,\Delta,\varepsilon,\phi) \mapsto (m,u,\Delta,\varepsilon,\id) \]
    and provides a section for the the forgetful functor $F_1$ from \eqref{eq:diag}, which is thus faithful.
\end{prop}
\begin{proof}
    The assignment respects the relations of $\OblCob$ from Proposition~\ref{prop:OblCob}.
\end{proof}

\renewcommand*{\bibfont}{\small}
\setlength{\bibitemsep}{0pt}
\raggedright
\printbibliography

\end{document}
\typeout{get arXiv to do 4 passes: Label(s) may have changed. Rerun}